\documentclass{amsart}
\usepackage{amssymb,mathtools}
\usepackage[british]{babel}
\usepackage{enumitem}
\usepackage[latin1]{inputenc}
\usepackage{url}
\usepackage{longtable}
\usepackage{hyperref}

\newtheorem{theorem}{Theorem}
\numberwithin{theorem}{section}
\newtheorem{corollary}[theorem]{Corollary}
\newtheorem{lemma}[theorem]{Lemma}

\theoremstyle{definition}
\newtheorem{definition}[theorem]{Definition}
\newtheorem{remark}[theorem]{Remark}

\newcommand{\rca}{\mathbf{RCA}}
\newcommand{\aca}{\mathbf{ACA}}
\newcommand{\atr}{\mathbf{ATR}}
\newcommand{\supp}{\operatorname{supp}}

\newcommand{\hau}{\mathbf{H}}

\title{Predicative collapsing principles}
\author{Anton Freund}

\begin{document}

\begin{abstract}
We show that arithmetical transfinite recursion is equivalent to a suitable formalization of the following: For every ordinal $\alpha$ there exists an ordinal $\beta$ such that $1+\beta\cdot(\beta+\alpha)$ (ordinal arithmetic) admits an almost order preserving collapse into~$\beta$. Arithmetical comprehension is equivalent to a statement of the same form, with $\beta\cdot\alpha$ at the place of $\beta\cdot(\beta+\alpha)$. We will also characterize the principles that any set is contained in a countable coded $\omega$-model of arithmetical transfinite recursion resp.~arithmetical comprehension.
\end{abstract}

\keywords{Reverse mathematics, well-ordering principles, Bachmann-Howard fixed points, dilators, Veblen function, arithmetical transfinite recursion, arithmetical comprehension}
\subjclass[2010]{03B30, 03F15, 03F35}

\maketitle
{\let\thefootnote\relax\footnotetext{\copyright~2020 Association for Symbolic Logic. This is the accepted version of a paper published in \emph{The Journal of Symbolic Logic} 85(1) 2020, pp. 511-530, \href{https://doi.org/10.1017/jsl.2019.83}{doi:10.1016/j.aim.2019.106767}.}}

\section{Introduction}

Well-ordering principles (of type one) are statements which assert that ``$T(X)$~is well-founded for any well-order~$X$", for some transformation $T$ of linear orders. We will consider such statements from the viewpoint of reverse mathematics (see~\cite{simpson09} for a comprehensive introduction). In this setting $X$ ranges over ordered subsets of~$\mathbb N$. The fact that $T$ is a transformation of linear orders can usually be proved in~$\rca_0$, so that the entire strength of the well-ordering principle lies in the preservation of well-foundedness.

The literature contains many results that characterize important $\Pi^1_2$-statements in terms of well-ordering principles. In order to explain our approach we focus on the following equivalence (but further results will be covered below):
\begin{theorem}[H.~Friedman {[unpublished]}; M.~Rathjen and A.~Weiermann~\cite{rathjen-weiermann-atr}]\label{thm:friedman-atr}
 The following are equivalent over $\rca_0$:
 \begin{enumerate}[label=(\roman*)]
  \item arithmetical transfinite recursion (i.\,e.~the principal axiom of $\atr_0$),
  \item the statement that $\varphi (1+X)0$ is well-founded for any well-order~$X$.
 \end{enumerate}
\end{theorem}
The transformation in~(ii) is related to the Veblen function, which iterates derivatives of normal functions into the transfinite (cf.~\cite{schuette77,marcone-montalban}). In the context of reverse mathematics, the relevant values of this function can be represented by relativized ordinal notation systems~$\varphi (1+X)0$ (see~\cite[Definition~2.2]{rathjen-weiermann-atr} for details; our summand $1$ corresponds to the minimal element $0_Q$ in the cited definition).

The present paper shows that complicated well-ordering principles can, in a certain sense, be reduced to much simpler ones. In particular we will reduce the well-ordering principle $X\mapsto\varphi (1+X)0$ to the family of order transformations
\begin{equation*}
Y\mapsto T^\varphi_X(Y):=1+(Y+X)\times Y,
\end{equation*}
indexed by the order $X$ (unless stated otherwise, ``order" will always mean ``linear order"). Here $1$ denotes the order with a single element. Also recall that the sum of two orders $X=(X,<_X)$ and $Y=(Y,<_Y)$ has underlying set
\begin{align*}
X+Y=\{\langle 0,x\rangle\,|\,x\in X\}\cup\{\langle1,y\rangle\,|\,y\in Y\}.
\end{align*}
For $x<_X x'$ and $y<_Y y'$ we have $\langle 0,x\rangle<_{X+Y}\langle 0,x'\rangle$ resp.~$\langle 1,y\rangle<_{X+Y}\langle 1,y'\rangle$, and $\langle 0,x\rangle<_{X+Y}\langle 1,y\rangle$ holds for any $x\in X$ and $y\in Y$. The product is given by
\begin{equation*}
X\times Y=\{\langle x,y\rangle\,|\,x\in X\text{ and }y\in Y\},
\end{equation*}
where $\langle x,y\rangle<_{X\times Y}\langle x',y'\rangle$ holds if we have $x<_X x'$, or $x=x'$ and $y<_Yy'$. Clearly the definition of sum and product is much simpler than the construction of $\varphi (1+X)0$ in~\cite[Definition~2.2]{rathjen-weiermann-atr}. The fact that sums and products of well-orders are themselves well-ordered can be proved in $\rca_0$, in contrast to Theorem~\ref{thm:friedman-atr}.

So how can $X\mapsto\varphi (1+X)0$ be reduced to the transformations $T^\varphi_X$? The idea is to consider fixed points of a certain type. Let us first observe that $T^\varphi_X(Y)\cong Y$ cannot hold for any well-orders $X$ and $Y$: If the latter have order types $\alpha$ resp.~$\beta$, then $T^\varphi_X(Y)$ has order type $1+\beta\cdot(\beta+\alpha)>\beta$. The best we can hope for is an ``almost'' order preserving function
\begin{equation*}
\vartheta:T^\varphi_X(Y)\rightarrow Y.
\end{equation*}
To make this precise we need some terminology: A transformation $Y\mapsto T(Y)$ of linear orders is called inclusive if $T(Y_0)$ is a suborder of $T(Y)$ whenever $Y_0$ is a suborder of $Y$. This property allows us to introduce the following notion:

\begin{definition}\label{def:supports}
 Let $Y\mapsto T(Y)$ be an inclusive transformation of orders. Given any order $Y$, we define the support $\supp^T_Y(\sigma)\subseteq Y$ of an element $\sigma\in T(Y)$ by
\begin{equation*}
 \supp^T_Y(\sigma)=\bigcap\{Y_0\subseteq Y\,|\,\sigma\in T(Y_0)\}.
\end{equation*}
\end{definition}
For the above transformations $T^\varphi_X$, the supports have a concrete description: The element of the summand $1$ has empty support. The support of an element $\langle\langle 0,y\rangle,y'\rangle$ resp.~$\langle\langle 1,x\rangle,y'\rangle$ in the other summand is equal to $\{y,y'\}$ resp.~$\{y'\}$. We can now say what we mean by an ``almost'' order preserving function:
\begin{definition}\label{def:collapse}
 Consider an inclusive transformation $Y\mapsto T(Y)$ of linear orders. A function $\vartheta:T(Y)\rightarrow Y$ is called a Bachmann-Howard collapse if the following holds for all $\sigma,\tau\in T(Y)$:
 \begin{enumerate}[label=(\roman*)]
  \item $\sigma<_{T(Y)}\tau$ implies $\vartheta(\sigma)<_Y\vartheta(\tau)$, under the side condition that $y<_Y\vartheta(\tau)$ holds for all $y\in\supp^T_Y(\sigma)$,
  \item we have $y<_Y\vartheta(\sigma)$ for all $y\in\supp^T_Y(\sigma)$.
 \end{enumerate}
 An order~$Y$ that admits such a function is called a Bachmann-Howard fixed point of~$T$. If $Y$ can be embedded into any other Bachmann-Howard fixed point of~$T$, then it is called a minimal Bachmann-Howard fixed point.
\end{definition}
In Remark~\ref{rmk:minimal-initial} we will discuss a stronger notion of minimality, which may be more appealing from a categorical standpoint. We can now state our characterization of the transformation $X\mapsto\varphi (1+X)0$, which will be proved in Section~\ref{sect:veblen}.

\begin{theorem}[$\rca_0$]\label{thm:veblen-collapsing}
The order $\varphi(1+X)0$ is a minimal Bachmann-Howard fixed point of the transformation $Y\mapsto 1+(Y+X)\times Y$, for any linear order~$X$.
\end{theorem}

Due to minimality, a descending sequence in $\varphi(1+X)0$ propagates to any Bachmann-Howard fixed point of~$T^\varphi_X$. Hence $\varphi(1+X)0$ is well-founded if, and only if, the transformation $T^\varphi_X$ has a well-founded Bachmann-Howard fixed point. Together with Theorem~\ref{thm:friedman-atr} we obtain the following:

\begin{samepage}
\begin{corollary}\label{cor:exp-atr}
The following principles are equivalent over $\rca_0$:
 \begin{enumerate}[label=(\roman*)]
  \item arithmetical transfinite recursion,
  \item for every well-order~$X$ the transformation $Y\mapsto 1+(Y+X)\times Y$ has a well-founded Bachmann-Howard fixed point.
 \end{enumerate}
\end{corollary}
\end{samepage}

As mentioned above, the literature contains several results that have the same form as Theorem~\ref{thm:friedman-atr}. In each line of the following table, the principle in the left column is equivalent to the assertion that the transformation in the middle column preserves well-foundedness (so the third line is Theorem~\ref{thm:friedman-atr}). Precise definitions and proofs can be found in the references that are given in the right column.
{\def\arraystretch{1.75}\tabcolsep=16pt
\setlength{\LTpre}{\baselineskip}\setlength{\LTpost}{\baselineskip}
 \begin{longtable}{lll}\hline
  arithmetical comprehension & $X\mapsto\omega^X$ & \cite{girard87,hirst94}\\ \hline
  the $\omega$-jump of every set exists & $X\mapsto\varepsilon_X$ & \cite{marcone-montalban,rathjen-afshari}\\ \hline
  arithmetical transfinite recursion & $X\mapsto\varphi (1+X)0$ & \cite{rathjen-weiermann-atr}\\ \hline
  every set lies in an $\omega$-model of $\atr$ & $X\mapsto\Gamma_X$ & \cite{rathjen-atr}\\ \hline
 \end{longtable}}
 Note that the existence of $\omega$-jumps is equivalent to the statement that every set lies in a (countable coded) $\omega$-model of $\aca$, over the base theory $\aca_0$ (see~\cite[Lemma~3.4]{rathjen-afshari}; $\omega$-models are explained in~\cite[Section~7.2]{simpson09}). We will characterize all transformations from the previous table in terms of collapsing principles. In each line of the next table, the order in the left column is a minimal Bachmann-Howard fixed point of the transformation in the middle column, for any linear order~$X$. The right column refers to the corresponding theorem of the present paper.
{\def\arraystretch{1.75}\tabcolsep=16pt
\setlength{\LTpre}{\baselineskip}\setlength{\LTpost}{\baselineskip}
 \begin{longtable}{lll}\hline
  $\omega^{\omega^X}$ & $Y\mapsto T^\omega_X(Y):=1+(1+X)\times Y$ & Theorem~\ref{thm:char-exp}\\ \hline
  $\varepsilon_X$ & $Y\mapsto T^\varepsilon_X(Y):=1+Y^2+X$ & Theorem~\ref{thm:eps-fp}\\ \hline
  $\varphi(1+X)0$ & $Y\mapsto T^\varphi_X(Y):=1+(Y+X)\times Y$ & Theorem~\ref{thm:veblen-collapsing}\\ \hline
  $\Gamma_X$ & $Y\mapsto T^\Gamma_X(Y):=1+2\times Y^2+X$ & Theorem~\ref{thm:Gamma}\\ \hline
 \end{longtable}}
 As in Corollary~\ref{cor:exp-atr}, we obtain new characterizations of the (broadly) predicative principles from above: arithmetical comprehension (Corollary~\ref{cor:arith-comp-collapsing}), the existence of $\omega$-jumps (Corollary~\ref{cor:omega-jumps}), and the existence of $\omega$-models of $\atr$ (Corollary~\ref{cor:model-atr}).

In the rest of this introduction we explain the wider context of our results: Let us first recall that J.-Y.~Girard~\cite{girard-pi2} has singled out a class of particularly uniform well-ordering principles, which are known as dilators. More precisely, a dilator is an endofunctor on the category of well-orders that preserves direct limits and pullbacks. In the inclusive case, these requirements correspond to the following properties of the supports from Definition~\ref{def:supports} (cf.~\cite[Remark~2.2.2]{freund-thesis}):
\begin{itemize}
 \item each support $\supp^T_Y(\sigma)\subseteq Y$ is finite,
 \item we have $\sigma\in T(\supp^T_Y(\sigma))$ for any $\sigma\in T(Y)$.
\end{itemize}
Girard has shown that dilators are determined by their restrictions to the category of natural numbers (up to natural equivalence). This crucial property makes it possible to represent dilators in second order arithmetic. For the purpose of the present paper we do not need this general representation, since our families of order transformations come with an explicit parametrization.
 
The notion of Bachmann-Howard fixed point has been introduced in~\cite{freund-thesis,freund-equivalence}, for arbitrary (i.\,e.~not necessarily inclusive) dilators. In the cited papers it was shown that \mbox{$\Pi^1_1$-comprehension} is equivalent to the statement that every dilator has a well-founded Bachmann-Howard fixed point. Furthermore, a minimal Bachmann-Howard fixed point of a given dilator can already be constructed in~$\rca_0$, as shown in~\cite{freund-categorical,freund-computable}. Due to its minimality, that fixed point must be well-founded, but $\rca_0$ cannot prove this fact. Applied to Corollary~\ref{cor:exp-atr}, this confirms that the strength of statement~(ii) does not lie in the existence of a Bachmann-Howard fixed point as such, but rather in its well-foundedness.

The name ``Bachmann-Howard fixed point" refers to the fact that our definitions are inspired by the Bachmann-Howard ordinal, in particular by the notation system from~\cite[Section~1]{rathjen-weiermann-kruskal}. It is well-known that values of the Veblen function also arise in the construction of the Bachmann-Howard ordinal (see e.\,g.~\cite{buchholz-bachmann}). For this reason a result such as Theorem~\ref{thm:veblen-collapsing} may not be entirely unexpected. Nevertheless, it seems that the connection on the level of predicative well-ordering principles has not been made before. The literature does contain an impredicative well-ordering principle that is related to the Bachmann-Howard ordinal: As shown by M.~Rathjen and P.~Valencia Vizca\'ino~\cite{rathjen-model-bi}, the statement that every set lies in an $\omega$-model of bar induction is equivalent to the principle that a relativized notation system $\vartheta_X$ is well-founded for any well-order~$X$. In contrast to our approach, the notation system $\vartheta_X$ incorporates the collapsing function into the term structure.

In the present paper we are concerned with ``almost" order preserving collapsing functions of transformations that do not have well-founded fixed points in the usual sense. A class of transformations that correspond to normal functions has been singled out by P.~Aczel~\cite{aczel-phd,aczel-normal-functors}: these transformations have well-founded fixed points of arbitrarily large order type. In~\cite{freund-rathjen_derivatives,freund_single-fp} it was shown that an appropriate formalization of the statement that ``every normal function has a derivative (resp.~at least one fixed point)" is equivalent to $\Pi^1_1$-induction along arbitrary well-orders (resp.~along~$\mathbb N$). These induction principles are considerably weaker than the principle of $\Pi^1_1$-comprehension, which is equivalent to the existence of well-founded Bachmann-Howard fixed points. The present paper appears to show that the great strength of Bachmann-Howard fixed points translates into particularly simple characterizations of weaker principles.

\section{Collapsing and ordinal exponentiation}\label{sect:exp}

In the present section we show how the orders $\omega^{\omega^X}$ and $\varepsilon_X$ can be constructed as Bachmann-Howard fixed points. As mentioned in the introduction, this yields characterizations of arithmetical comprehension and the principle that the $\omega$-jump of every set exists.

Let us recall some definitions: Given an order $X=(X,\leq_X)$, the underlying set
\begin{equation*}
\omega^X=\{\langle x_0,\dots,x_{n-1}\rangle\,|\,x_{n-1}\leq_X\dots\leq_X x_0\}
\end{equation*}
of the order $\omega^X=(\omega^X,\leq_{\omega^X})$ consists of the finite decreasing sequences with entries from~$X$. The relation $\leq_{\omega^X}$ is defined as the lexicographic order on this set (cf.~\cite[Definition~2.2]{hirst94}). Intuitively, the elements of $\omega^X$ correspond to Cantor normal forms. To convey this intuition we will write $\omega^{x_0}+\dots+\omega^{x_{n-1}}$ rather than~$\langle x_0,\dots,x_{n-1}\rangle$, and in particular $0$ rather than $\langle\rangle\in\omega^X$. If $X$ is a well-order of type~$\alpha$, then $\omega^X$ has order type~$\omega^\alpha$, in the usual sense of ordinal arithmetic.

Addition on $\omega^X$ can be defined in terms of Cantor normal forms: We agree that~$0$ is neutral and that we have
\begin{equation*}
(\omega^{x_0}+\dots+\omega^{x_n})+(\omega^{y_0}+\dots+\omega^{y_m})=\omega^{x_0}+\dots+\omega^{x_i}+\omega^{y_0}+\dots+\omega^{y_m},
\end{equation*}
where $i$ is maximal with $y_0\leq_X x_i$ (note $i=-1$ if $x_0<_X y_0$). It is well-known that basic properties of ordinal addition can be proved in $\rca_0$ (cf.~e.\,g.~\cite{schuette77,sommer95}).

In order to define multiplication we must consider $\omega^{\omega^X}$ rather than $\omega^X$ (note that an ordinal of the form $\omega^\alpha$ does not need to be multiplicatively principal). The general definition of multiplication in terms of Cantor normal forms is somewhat cumbersome, since ordinal arithmetic does not validate right distributivity. Luckily, we will only need to multiply terms of a particular form: Given elements $\alpha\in\omega^X$ and $\beta_{n-1}\leq_{\omega^X}\dots\leq_{\omega^X}\beta_0$, we can set
\begin{equation*}
\omega^\alpha\cdot(\omega^{\beta_0}+\dots+\omega^{\beta_{n-1}})=\omega^{\alpha+\beta_0}+\dots+\omega^{\alpha+\beta_{n-1}},
\end{equation*}
where the exponents are added in~$\omega^X$. Since $\omega^X$ contains a minimal element $0=\langle\rangle$, the order $\omega^{\omega^X}$ contains a minimal non-zero element $1=\omega^0$. This allows us to distinguish between successor and limit elements. Again, basic properties of these notions can be proved in $\rca_0$. To avoid iterated superscripts we will abbreviate $\omega_2(X):=\omega^{\omega^X}$, as well as $\omega_2(x)=:\omega^{\omega^x}\in\omega_2(X)$ for $x\in X$.

On an intuitive level one would like to prove certain statements by induction along the order $\leq_{\omega_2(X)}$, but this induction principle is not available in our setting. Instead we argue by induction over the length of terms. For this purpose we define functions \mbox{$l^\omega_X:\omega^X\rightarrow\mathbb N$} and $L^\omega_X:\omega_2(X)\rightarrow\mathbb N$ by setting
\begin{align*}
l^\omega_X(\omega^{x_0}+\dots+\omega^{x_{n-1}})&=n,\\
L^\omega_X(\omega^{\alpha_0}+\dots+\omega^{\alpha_{n-1}})&=l^\omega_X(\alpha_0)+\dots+l^\omega_X(\alpha_{n-1})+n.
\end{align*}
The following observation will be crucial for our analysis of the order $\omega_2(X)$.
\begin{lemma}[$\rca_0$]\label{lem:decompose-omega2x}
Let $X$ be a linear order. Any limit element of $\omega_2(X)$ can be uniquely written as $\omega_2(x)\cdot\eta$ with $0<_{\omega_2(X)}\eta<_{\omega_2(X)}\omega_2(x)\cdot\eta$. Furthermore we have $L^\omega_X(\eta)<L^\omega_X(\omega_2(x)\cdot\eta)$ for any such decomposition.
\end{lemma}
\begin{proof}
To establish existence we consider an arbitrary limit element
\begin{equation*}
\omega^{\beta_0}+\dots+\omega^{\beta_n}\in\omega_2(X).
\end{equation*}
Since we are concerned with a limit, the last exponent is different from $0\in\omega^X$. Hence there are elements $x\in X$ and $\gamma_n\in\omega^X$ with
\begin{equation*}
\beta_n=\omega^x+\gamma_n>\gamma_n.
\end{equation*}
Let us also record $l^\omega_X(\gamma_n)<l^\omega_X(\beta_n)$. Left subtraction is readily defined on the level of Cantor normal forms. In view of $\omega^x\leq_{\omega^X}\beta_n\leq_{\omega^X}\dots\leq_{\omega^X}\beta_0$ we can thus write
\begin{equation*}
\beta_i=\omega^x+\gamma_i
\end{equation*}
for all $i<n$. One readily checks $\gamma_i\leq_{\omega^X}\beta_i$ and $l^\omega_X(\gamma_i)\leq l^\omega_X(\beta_i)$ (note that the inequalities may not be strict for~$i<n$). Due to the monotonicity of addition we must also have $\gamma_n\leq_{\omega^X}\dots\leq_{\omega^X}\gamma_0$. We can thus define
\begin{equation*}
\eta:=\omega^{\gamma_0}+\dots+\omega^{\gamma_n}\in\omega_2(X).
\end{equation*}
By construction we have
\begin{equation*}
\omega_2(x)\cdot\eta=\omega^{\omega^x+\gamma_0}+\dots+\omega^{\omega^x+\gamma_n}=\omega^{\beta_0}+\dots+\omega^{\beta_n}.
\end{equation*}
The above inequalities between $\beta_i$ and $\gamma_i$ imply
\begin{equation*}
0<_{\omega_2(X)}\eta<_{\omega_2(X)}\omega^{\beta_0}+\dots+\omega^{\beta_n}=\omega_2(x)\cdot\eta.
\end{equation*}
In view of  $l^\omega_X(\gamma_n)<l^\omega_X(\beta_n)$ we also get
\begin{equation*}
L^\omega_X(\eta)=l^\omega_X(\gamma_0)+\dots+l^\omega_X(\gamma_n)+n+1<l^\omega_X(\beta_0)+\dots+l^\omega_X(\beta_n)+n+1=L^\omega_X(\omega_2(x)\cdot\eta).
\end{equation*}
It remains to establish uniqueness. Due to the monotonicity of multiplication it suffices to show that
\begin{equation*}
\eta<_{\omega_2(X)}\omega_2(x)\cdot\eta=\omega_2(y)\cdot\xi>_{\omega_2(X)}\xi
\end{equation*}
implies $x=y$. Aiming at a contradiction, let us assume that we have $x<_X y$. Then we get $\omega^x+\omega^y=\omega^y$ and hence $\omega_2(x)\cdot\omega_2(y)=\omega_2(y)$. We can deduce
\begin{equation*}
\omega_2(x)\cdot\eta=\omega_2(y)\cdot\xi=\omega_2(x)\cdot\omega_2(y)\cdot\xi=\omega_2(x)\cdot\omega_2(x)\cdot\eta,
\end{equation*}
which is incompatible with the assumption $\eta<_{\omega_2(X)}\omega_2(x)\cdot\eta$.
\end{proof}

Our goal is to characterize $\omega_2(X)$ as a minimal Bachmann-Howard fixed point of the order transformation
\begin{equation*}
 Y\mapsto T^\omega_X(Y)=1+(1+X)\times Y.
\end{equation*}
Let us write $\bot$ for the unique element of $1$. The elements of $T^\omega_X(Y)$ will be written as $\bot$, $\langle\bot,y\rangle$ and $\langle 1+x,y\rangle$, rather than $\langle 0,\bot\rangle$, $\langle 1,\langle\langle 0,\bot\rangle,y\rangle\rangle$ and $\langle 1,\langle\langle 1,x\rangle,y\rangle\rangle$, respectively. Sometimes we also use $x$ to denote an arbitrary element of $1+X$. The supports from Definition~\ref{def:supports} take the forms
\begin{equation*}
 \supp^\omega_Y(\bot)=\emptyset\qquad\text{and}\qquad\supp^\omega_Y(\langle x,y\rangle)=\{y\}.
\end{equation*}
In view of Definition~\ref{def:collapse}, this means that a function $\vartheta:T^\omega_X(Y)\rightarrow Y$ is a Bachmann-Howard collapse if, and only if, the following conditions are satisfied:
\begin{enumerate}
 \item[(i)] we have $\vartheta(\bot)<_Y\vartheta(\langle x,y\rangle)$ for any $\langle x,y\rangle\in (1+X)\times Y$,
 \item[(i$'$)] $\langle x,y\rangle<_{(1+X)\times Y}\langle x',y'\rangle$ implies $\vartheta(\langle x,y\rangle)<_Y\vartheta(\langle x',y'\rangle)$, under the side condition that we have $y<_Y\vartheta(\langle x',y'\rangle)$,
 \item[(ii)] we have $y<_Y\vartheta(\langle x,y\rangle)$ for any $x\in 1+X$ and $y\in Y$.
\end{enumerate}
We can now establish the promised characterization, improving~\cite[Proposition~3.3]{freund-computable}:

\begin{theorem}[$\rca_0$]\label{thm:char-exp}
The order $\omega^{\omega^X}$ is a minimal Bachmann-Howard fixed point of the transformation $T^\omega_X$, for any order $X$.
\end{theorem}
\begin{proof}
In order to show that $\omega_2(X)=\omega^{\omega^X}$ is a Bachmann-Howard fixed point of~$T^\omega_X$ we must define a collapsing function
\begin{equation*}
\vartheta:1+(1+X)\times\omega_2(X)\rightarrow\omega_2(X).
\end{equation*}
Using the successor operation and multiplication in $\omega_2(X)$, we set
\begin{align*}
\vartheta(\bot)&:=0,\\
\vartheta(\langle \bot,\eta\rangle)&:=\eta+1,\\
\vartheta(\langle 1+x,\eta\rangle)&:=\omega_2(x)\cdot(\eta+1).
\end{align*}
The above condition~(i) is immediate. Condition~(ii) is satisfied in view of
\begin{equation*}
 \eta<_{\omega_2(X)}\eta+1\leq_{\omega_2(X)}\omega_2(x)\cdot(\eta+1).
\end{equation*}
To verify condition~(i$'$) one needs to distinguish several cases. In the first interesting case we are concerned with an inequality
\begin{equation*}
 \langle\bot,\eta\rangle<_{(1+X)\times Y}\langle 1+x,\eta'\rangle.
\end{equation*}
Due to the side condition in~(i$'$) we may assume
\begin{equation*}
 \eta<_{\omega_2(X)}\vartheta(\langle 1+x,\eta'\rangle)=\omega_2(x)\cdot(\eta'+1).
\end{equation*}
The element on the right side is a limit (note that the last exponent in its Cantor normal form is equal to $\omega^x\neq 0$). Hence we obtain
\begin{equation*}
 \vartheta(\langle\bot,\eta\rangle)=\eta+1<_{\omega_2(X)}\vartheta(\langle 1+x,\eta'\rangle),
\end{equation*}
as required. Let us also consider the case of an inequality
\begin{equation*}
 \langle 1+x,\eta\rangle<_{(1+X)\times Y}\langle 1+x',\eta'\rangle
\end{equation*}
with $x<_X x'$. Yet again, the side condition yields $\eta+1<_{\omega_2(X)}\omega_2(x')\cdot(\eta'+1)$. Also observe that $x<_X x'$ implies $\omega_2(x)\cdot\omega_2(x')=\omega_2(x')$, as in the proof of Lemma~\ref{lem:decompose-omega2x}. Using the monotonicity of multiplication we can deduce
\begin{align*}
 \vartheta(\langle 1+x,\eta\rangle)=\omega_2(x)\cdot(\eta+1)<_{\omega_2(X)}\omega_2(x)\cdot\omega_2(x')\cdot(\eta'+1)&=\\
 =\omega_2(x')\cdot(\eta'+1)&=\vartheta(\langle 1+x',\eta'\rangle).
\end{align*}
So far we have shown that $\omega_2(X)$ is a Bachmann-Howard fixed point of $T^\omega_X$. To establish minimality we consider an arbitrary order~$Y$ that admits a Bachmann-Howard collapse
\begin{equation*}
 \vartheta:1+(1+X)\times Y\rightarrow Y.
\end{equation*}
We need to construct an embedding $f:\omega_2(X)\rightarrow Y$. In view of Lemma~\ref{lem:decompose-omega2x} we can define $f$ by recursion over the length of terms, by setting
\begin{align*}
 f(0)&:=\vartheta(\bot),\\
 f(\eta+1)&:=\vartheta(\langle\bot,f(\eta)\rangle),\\
 f(\omega_2(x)\cdot\eta)&:=\vartheta(\langle 1+x,f(\eta)\rangle),\quad\text{where $0<_{\omega_2(X)}\eta<_{\omega_2(X)}\omega_2(x)\cdot\eta$}.
\end{align*}
To show that $\xi<_{\omega_2(X)}\xi'$ implies $f(\xi)<_Y f(\xi')$ we argue by induction on the combined length $L^\omega_X(\xi)+L^\omega_X(\xi')$ of $\xi$ and $\xi'$ (note that this amounts to an induction over a $\Pi^0_1$-statement, which is available in $\rca_0$). In the first interesting case we consider an inequality
\begin{equation*}
 \xi=\eta+1<_{\omega_2(X)}\omega_2(x)\cdot\eta'=\xi'.
\end{equation*}
In view of $\bot<_{1+X}1+x$ we clearly have $\langle\bot,f(\eta)\rangle<_{(1+X)\times Y}\langle 1+x,f(\eta')\rangle$. Invoking the induction hypothesis, we also see that $\eta<_{\omega_2(X)}\xi'$ implies
\begin{equation*}
 f(\eta)<_Y f(\xi')=\vartheta(\langle 1+x,f(\eta')\rangle).
\end{equation*}
This is the side condition required in clause~(i$'$) above. We can thus conclude
\begin{equation*}
 f(\xi)=\vartheta(\langle\bot,f(\eta)\rangle)<_Y\vartheta(\langle 1+x,f(\eta')\rangle)=f(\xi').
\end{equation*}
Let us now consider an inequality of the form
\begin{equation*}
 \xi=\omega_2(x)\cdot\eta<_{\omega_2(X)}\eta'+1=\xi'.
\end{equation*}
Using the induction hypothesis and clause~(ii) above we get
\begin{equation*}
 f(\xi)\leq_Y f(\eta')<_Y\vartheta(\langle\bot,f(\eta')\rangle)=f(\xi'),
\end{equation*}
as required. To conclude the proof we consider an inequality of the form
\begin{equation*}
 \xi=\omega_2(x)\cdot\eta<_{\omega_2(X)}\omega_2(x')\cdot\eta'=\xi'.
\end{equation*}
We need to distinguish three cases: First assume that we have $x<_X x'$. Then we immediately get $\langle 1+x,f(\eta)\rangle<_{(1+X)\times Y}\langle 1+x',f(\eta')\rangle$. In view of Lemma~\ref{lem:decompose-omega2x} we have $\eta<_{\omega_2(X)}\xi<_{\omega_2(X)}\xi'$ and $L^\omega_X(\eta)<L^\omega_X(\xi)$. Hence the induction hypothesis yields
\begin{equation*}
 f(\eta)<_Y f(\xi')=\vartheta(\langle 1+x',f(\eta')\rangle).
\end{equation*}
Hence the side condition from clause~(i$'$) is satisfied, and we obtain
\begin{equation*}
 f(\xi)=\vartheta(\langle 1+x,f(\eta)\rangle)<_Y\vartheta(\langle 1+x',f(\eta')\rangle)=f(\xi').
\end{equation*}
Now assume $x=x'$. In view of $\xi<_{\omega_2(X)}\xi'$ we must have $\eta<_{\omega_2(X)}\eta'$. Then the induction hypothesis yields $\langle 1+x,f(\eta)\rangle<_{(1+X)\times Y}\langle 1+x',f(\eta')\rangle$, and we can conclude as in the previous case. Finally assume $x>_X x'$. In this case we observe that
\begin{equation*}
 \omega_2(x')\cdot\xi=\xi<_{\omega_2(X)}\xi'=\omega_2(x')\cdot\eta'
\end{equation*}
implies $\xi<_{\omega_2(X)}\eta'$. Using the induction hypothesis and clause~(ii) we obtain
\begin{equation*}
 f(\xi)<_Y f(\eta')<_Y\vartheta(\langle 1+x',f(\eta')\rangle)=f(\xi'),
\end{equation*}
just as needed.
\end{proof}

The statement that $\omega^X$ is well-founded for every well-order~$X$ is equivalent to arithmetical comprehension, as shown by J.-Y.~Girard~\cite[Section~5.4]{girard87} (cf.~also the computability-theoretic proof by J.~Hirst~\cite{hirst94}). We can deduce the following:

\begin{corollary}\label{cor:arith-comp-collapsing}
The following are equivalent over $\rca_0$:
\begin{enumerate}[label=(\roman*)]
 \item arithmetical comprehension (which is the principal axiom of $\aca_0$),
 \item for every well-order~$X$ the transformation $Y\mapsto 1+X\times Y$ has a well-founded Bachmann-Howard fixed point.
\end{enumerate}
\end{corollary}
\begin{proof}
 To deduce (ii) from~(i) we consider an arbitrary well-order~$X$. In view of Girard's result, we can invoke~(i) to infer that $\omega^X$ and $\omega_2(X)$ are well-founded. Theorem~\ref{thm:char-exp} yields a Bachmann-Howard collapse
 \begin{equation*}
  \vartheta:1+(1+X)\times\omega_2(X)\rightarrow\omega_2(X).
 \end{equation*}
 The restriction of $\vartheta$ to $1+X\times\omega_2(X)$ witnesses that $\omega_2(X)$ is a Bachmann-Howard fixed point of $Y\mapsto 1+X\times Y$, as one readily verifies. To show that~(ii) implies~(i) we again invoke Girard's result. Hence we must establish that $\omega^X$ is well-founded for an arbitrary well-order~$X$. Since $1+X$ is still well-founded, we can use~(ii) to get a well-founded Bachmann-Howard fixed point $Y$ of the transformation
 \begin{equation*}
  Y\mapsto 1+(1+X)\times Y=T^\omega_X(Y).
 \end{equation*}
 From Theorem~\ref{thm:char-exp} we know that $\omega_2(X)$ can be embedded into~$Y$. Hence $\omega_2(X)$ must be well-founded as well. In view of the embedding
 \begin{equation*}
  \omega^X\ni\alpha\mapsto\omega^\alpha\in\omega_2(X)
 \end{equation*}
 we can infer that $\omega^X$ is well-founded, as required.
\end{proof}

To conclude the first half of the present section we discuss a possible improvement of Theorem~\ref{thm:char-exp}:

\begin{remark}\label{rmk:minimal-initial}
The Bachmann-Howard collapse $\vartheta:T^\omega_X(\omega_2(X))\rightarrow\omega_2(X)$ that we have constructed in the proof of Theorem~\ref{thm:char-exp} does not look quite optimal: For an element $\eta\in\omega_2(X)$ with $\eta<_{\omega_2(X)}\omega_2(x)\cdot\eta$ it might have been more natural to define $\vartheta(\langle 1+x,\eta\rangle)$ as $\omega_2(x)\cdot\eta$ rather than $\omega_2(x)\cdot(\eta+1)$. To make this intuition precise we can observe the following: In the second half of the proof of Theorem~\ref{thm:char-exp} we have constructed an embedding $f:\omega_2(X)\rightarrow Y$ into an arbitrary Bachmann-Howard fixed point $Y$ of the transformation $T^\omega_X$. If we construct this embedding with respect to the given Bachmann-Howard collapse for $Y=\omega_2(X)$, then we get
\begin{equation*}
f(\omega_2(x)\cdot\eta)=\vartheta(\langle 1+x,f(\eta)\rangle)=\omega_2(x)\cdot(f(\eta)+1),
\end{equation*}
which means that $f$ cannot be the identity on~$\omega_2(Y)$. In order to understand this phenomenon in general we recall that the notion of Bachmann-Howard fixed point was defined for dilators, i.\,e.~for particularly uniform endofunctors on the category of linear orders. Functoriality allows us to define the following notion: Given Bachmann-Howard fixed points $X$ and $Y$ with fixed collapsing functions $\vartheta_X:T(X)\rightarrow X$ and $\vartheta_Y:T(Y)\rightarrow Y$, we say that $f:X\rightarrow Y$ is a morphism of Bachmann-Howard fixed points if we have
\begin{equation*}
f\circ\vartheta_X=\vartheta_Y\circ T(f).
\end{equation*}
Following the usual categorical terminology, an initial Bachmann-Howard fixed point consists of an order~$X$ and a Bachmann-Howard collapse $\vartheta:T(X)\rightarrow X$ that admit a unique morphism into any Bachmann-Howard fixed point of the same dilator. The proofs of~\cite[Theorem~3.4]{freund-categorical} and~\cite[Theorem~4.5]{freund-computable} reveal that every dilator has an initial Bachmann-Howard fixed point, which is necessarily unique up to isomorphism. Note that any initial fixed point is minimal in the sense of Definition~\ref{def:collapse}. The notion of initial fixed point is certainly more satisfactory from a theoretical perspective. On the other hand, minimal fixed points are entirely sufficient to deduce Corollary~\ref{cor:arith-comp-collapsing} and similar results. We can also observe that the order type of a minimal fixed point is necessarily unique in the well-founded case. For these reasons we have decided to avoid the additional technicalities that would be necessary to determine initial fixed points, rather than just minimal ones.
\end{remark}

In the second half of this section we are concerned with the orders \mbox{$\varepsilon_X=(\varepsilon_X,<_{\varepsilon_X})$} that have been mentioned in the introduction. In contrast to the case of $\omega^X$, the set $\varepsilon_X$ and the relation $<_{\varepsilon_X}$ have to be defined simultaneously. The underlying set consists of the terms that are generated by the following clauses:
\begin{itemize}
 \item The set $\varepsilon_X$ contains a symbol~$0$, and a symbol $\varepsilon_x$ for each element $x\in X$.
 \item If $\alpha\in\varepsilon_X$ is not of the form $\varepsilon_x$, then we have a term $\omega^\alpha\in\varepsilon_X$.
 \item Given $n>1$ elements $\alpha_n\leq_{\varepsilon_X}\dots\leq_{\varepsilon_X}\alpha_1$ of $\varepsilon_X$, we get $\omega^{\alpha_1}+\dots+\omega^{\alpha_n}\in\varepsilon_X$.
\end{itemize}
The order $<_{\varepsilon_X}$ reflects the intuition that any term of the form $\varepsilon_x$ represents an $\varepsilon$-number, i.\,e.~an ordinal $\alpha$ that satisfies $\omega^\alpha=\alpha$. We refer to \cite[Definition~3.4]{freund-computable} for full details of the somewhat lengthy definition.

On the set $\varepsilon_X$ one can define counterparts of addition, multiplication and exponentiation to the base $\omega$, taking into account that $\varepsilon$-numbers are closed under these operations (cf.~\cite{schuette77}). In particular we have an operation
\begin{equation*}
\varepsilon_X\ni\alpha\mapsto\omega_2(\alpha):=\omega^{\omega^\alpha}\in\varepsilon_X,
\end{equation*}
which plays a similar (though somewhat less important) role as in the analysis of the order $\omega_2(X)$. To define a lenght function $L^\varepsilon_X:\varepsilon_X\rightarrow\mathbb N$ we recursively set
\begin{gather*}
 L^\varepsilon_X(0):=L^\varepsilon_X(\varepsilon_x):=0,\\
 L^\varepsilon_X(\omega^{\alpha_1}+\dots+\omega^{\alpha_n}):=L^\varepsilon_X(\alpha_0)+\dots+L^\varepsilon_X(\alpha_n)+n.
\end{gather*}
We say that an element of $\varepsilon_X$ is decomposable if it is neither equal to $0$ nor of the form $\varepsilon_x$. This terminology is justified in view of the following (cf.~Lemma~\ref{lem:decompose-omega2x}).

\begin{lemma}[$\rca_0$]\label{lem:eps-nf}
Any decomposable element of $\varepsilon_X$ can be uniquely written as $\omega^\alpha+\beta$ with $\alpha,\beta<_{\varepsilon_X}\omega^\alpha+\beta$. Furthermore we have $L^\varepsilon_X(\alpha),L^\varepsilon_X(\beta)<L^\varepsilon_X(\omega^\alpha+\beta)$ for any such decomposition.
\end{lemma}
\begin{proof}
 Let us first establish existence: Given a decomposable $\omega^{\alpha_1}+\dots+\omega^{\alpha_n}\in\varepsilon_X$, we set $\alpha:=\alpha_1$ and
 \begin{equation*}
  \beta:=\begin{cases}
         0 & \text{if $n=1$},\\
         \varepsilon_x & \text{if $n=2$ and $\alpha_2=\varepsilon_x$},\\
         \omega^{\alpha_2}+\dots+\omega^{\alpha_n} & \text{otherwise}.
        \end{cases}
 \end{equation*}
 By construction (and by the definition of addition and exponentiation on $\varepsilon_X$) we have $\omega^\alpha+\beta=\omega^{\alpha_1}+\dots+\omega^{\alpha_n}$. A straightforward induction on the term $\alpha_1$ yields
 \begin{equation*}
  \alpha_1<_{\varepsilon_X}\omega^{\alpha_1}+\dots+\omega^{\alpha_n},
 \end{equation*}
 which amounts to $\alpha<_{\varepsilon_X}\omega^\alpha+\beta$. In all cases it is straightforward to verify that we have $\beta<_{\varepsilon_X}\omega^\alpha+\beta$ as well as $L^\varepsilon_X(\alpha),L^\varepsilon_X(\beta)<L^\varepsilon_X(\omega^\alpha+\beta)$. Due to the monotonicity of addition, uniqueness reduces to the claim that
 \begin{equation*}
  \beta<_{\varepsilon_X}\omega^\alpha+\beta=\omega^\gamma+\delta>_{\varepsilon_X}\delta
 \end{equation*}
 implies $\alpha=\gamma$. Aiming at a contradiction, we assume $\alpha<_{\varepsilon_X}\gamma$. The latter yields
 \begin{equation*}
  \omega^\alpha+\beta=\omega^\gamma+\delta=\omega^\alpha+\omega^\gamma+\delta=\omega^\alpha+\omega^\alpha+\beta,
 \end{equation*}
 which is incompatible with $\beta<_{\varepsilon_X}\omega^\alpha+\beta$.
\end{proof}

We now want to characterize $\varepsilon_X$ as a minimal Bachmann-Howard fixed point of the order transformation
\begin{equation*}
 Y\mapsto T^\varepsilon_X(Y)=1+Y^2+X.
\end{equation*}
Elements of the summands $1$, $Y^2$ and $X$ will be written as $\bot$, $\langle y_0,y_1\rangle$ and $x$, respectively. The supports from Definition~\ref{def:supports} amount to
\begin{equation*}
 \supp^\varepsilon_Y(\bot)=\supp^\varepsilon_Y(x)=\emptyset\qquad\text{and}\qquad\supp^\varepsilon_Y(\langle y_0,y_1\rangle)=\{y_0,y_1\}.
\end{equation*}
Together with Definition~\ref{def:collapse}, this means that a function $\vartheta:T^\varepsilon_X(Y)\rightarrow Y$ is a Bachmann-Howard collapse if, and only if, the following conditions are satisfied:
\begin{enumerate}
 \item[(i)] we have $\vartheta(\bot)<_Y\vartheta(\langle y_0,y_1\rangle)$ for all $y_0,y_1\in Y$, and $\vartheta(\bot)<_Y\vartheta(x)<_Y\vartheta(x')$ for all $x,x'\in X$ with $x<_X x'$,
 \item[(i$'$)] $\langle y_0,y_1\rangle<_{Y^2}\langle y'_0,y'_1\rangle$ implies $\vartheta(\langle y_0,y_1\rangle)<_Y \vartheta(\langle y'_0,y'_1\rangle)$, under the side condition that we have $y_0,y_1<_Y\vartheta(\langle y'_0,y'_1\rangle)$,
 \item[(i$''$)] $\vartheta(\langle y_0,y_1\rangle)<_Y\vartheta(x)$ holds for any $y_0,y_1\in Y$ and $x\in X$ with $y_0,y_1<_Y\vartheta(x)$,
 \item[(ii)] we have $y_0,y_1<_Y\vartheta(\langle y_0,y_1\rangle)$ for all $y_0,y_1\in Y$.
\end{enumerate}
We can now establish the desired characterization:

\begin{theorem}[$\rca_0$]\label{thm:eps-fp}
The order $\varepsilon_X$ is a minimal Bachmann-Howard fixed point of the transformation~$T^\varepsilon_X$, for any order~$X$.
\end{theorem}
\begin{proof}
To witness that $\varepsilon_X$ is a Bachmann-Howard fixed point of $T^\varepsilon_X$ we need a collapsing function
\begin{equation*}
\vartheta:1+\varepsilon_X\times\varepsilon_X+X\rightarrow\varepsilon_X.
\end{equation*}
Relying on the ordinal arithmetic that is available in $\varepsilon_X$, we set
\begin{align*}
\vartheta(\bot)&:=0,\\
\vartheta(\langle\alpha,\beta\rangle)&:=\omega_2(\alpha+1)\cdot(\beta+1),\\
\vartheta(x)&:=\varepsilon_x.
\end{align*}
It is straightforward to see that the above conditions~(i) and~(ii) are satisfied (note that condition~(ii) could fail if we were to replace $\omega_2(\alpha+1)$ by $\omega_2(\alpha)$, as the proof of Theorem~\ref{thm:char-exp} might suggest). Condition~(i$'$) is verified as in the proof of~Theorem~\ref{thm:char-exp}. To establish condition~(i$''$) we consider arbitrary $\alpha,\beta\in\varepsilon_X$ and $x\in X$ with
\begin{equation*}
\alpha,\beta<_{\varepsilon_X}\vartheta(x)=\varepsilon_x.
\end{equation*}
Considering the order on $\varepsilon_X$ (cf.~\cite[Definition~3.4]{freund-computable}), it is straightforward to see that the element $\varepsilon_x$ behaves like an $\varepsilon$-number. Hence we obtain
\begin{equation*}
\vartheta(\langle\alpha,\beta\rangle)=\omega_2(\alpha+1)\cdot(\beta+1)<_{\varepsilon_X}\varepsilon_x=\vartheta(x).
\end{equation*}
This completes the proof that $\varepsilon_X$ is a Bachmann-Howard fixed point of~$T^\varepsilon_X$. Let us now consider an arbitrary Bachmann-Howard collapse
\begin{equation*}
\vartheta:1+Y^2+X\rightarrow Y.
\end{equation*}
We need to construct an embedding $f:\varepsilon_X\rightarrow Y$. In view of Lemma~\ref{lem:eps-nf} we can recursively define
\begin{align*}
f(0)&:=\vartheta(\bot),\\
f(\omega^\alpha+\beta)&:=\vartheta(\langle f(\alpha),f(\beta)\rangle),\quad\text{where $\alpha,\beta<\omega^\alpha+\beta$},\\
f(\varepsilon_x)&:=\vartheta(x).
\end{align*}
By induction on $L^\varepsilon_X(\eta)+L^\varepsilon_X(\xi)$ we can show that $\eta<_{\varepsilon_X}\xi$ implies $f(\eta)<_Y f(\xi)$. The first interesting case concerns an inequality
\begin{equation*}
\eta=\varepsilon_x<_{\varepsilon_X}\omega^\alpha+\beta=\xi.
\end{equation*}
Since $\varepsilon_x$ behaves like an $\varepsilon$-number we must have $\varepsilon_x\leq_{\varepsilon_X}\alpha$. Using the induction hypothesis and clause~(ii) above we get
\begin{equation*}
f(\eta)\leq_Y f(\alpha)<_Y \vartheta(\langle f(\alpha),f(\beta)\rangle)=f(\xi).
\end{equation*}
Let us now consider an inequality
\begin{equation*}
\eta=\omega^\alpha+\beta<_{\varepsilon_X}\varepsilon_x=\xi.
\end{equation*}
By Lemma~\ref{lem:eps-nf} we get $\alpha,\beta<_{\varepsilon_X}\xi$, so that the induction hypothesis yields
\begin{equation*}
f(\alpha),f(\beta)<_Y f(\xi)=\vartheta(x).
\end{equation*}
Invoking clause~(i$''$) we can infer
\begin{equation*}
f(\eta)=\vartheta(\langle f(\alpha),f(\beta)\rangle)<_{\varepsilon_X}\vartheta(x)=f(\xi).
\end{equation*}
Finally, we consider an inequality
\begin{equation*}
\eta=\omega^\alpha+\beta<_{\varepsilon_X}\omega^\gamma+\delta=\xi.
\end{equation*}
Considering the proof of Lemma~\ref{lem:eps-nf}, it is straightforward to see that we must have~$\alpha\leq_{\varepsilon_X}\gamma$. If we have $\alpha=\gamma$, then we get $\beta<_{\varepsilon_X}\delta$. In any case we can use the induction hypothesis to infer
\begin{equation*}
\langle f(\alpha),f(\beta)\rangle<_{Y^2}\langle f(\gamma),f(\delta)\rangle.
\end{equation*}
In view of $\alpha,\beta<_{\varepsilon_X}\eta<_{\varepsilon_X}\xi$ the induction hypothesis also yields
\begin{equation*}
f(\alpha),f(\beta)<_Y f(\xi)=\vartheta(\langle f(\gamma),f(\delta)\rangle).
\end{equation*}
By condition~(i$'$) we now obtain
\begin{equation*}
f(\eta)=\vartheta(\langle f(\alpha),f(\beta)\rangle)<_Y\vartheta(\langle f(\gamma),f(\delta)\rangle)=f(\xi),
\end{equation*}
as required.
\end{proof}

The statement that $\varepsilon_X$ is well-founded for any well-order~$X$ is equivalent to the assertion that the $\omega$-jump of any set exists, as shown by A.~Marcone and \mbox{A.~Montalb\'an}~\cite{marcone-montalban} (see also the proof-theoretic argument due to B.~Afshari and M.~Rathen~\cite{rathjen-afshari}). Together with Theorem~\ref{thm:eps-fp} we obtain the following:

\begin{corollary}\label{cor:omega-jumps}
The following are equivalent over $\rca_0$:
\begin{enumerate}[label=(\roman*)]
\item the $\omega$-jump of every set exists (which is the principal axiom of $\aca_0^+$),
\item for every well-order~$X$ the transformation $Y\mapsto 1+Y^2+X$ has a well-founded Bachmann-Howard fixed point.
\end{enumerate}
\end{corollary}

\section{Collapsing and the Veblen hierarchy}\label{sect:veblen}

In this section we show how the orders $\varphi(1+X)0$ and $\Gamma_X$ can be constructed as Bachmann-Howard fixed points. This will yield characterizations of arithmetical transfinite recursion and of the principle that every set lies in an $\omega$-model of $\atr$.

Let us begin by recalling the Veblen hierarchy: A function $f$ from ordinals to ordinals is called a normal function if it is strictly increasing and continuous at limit stages. Equivalently, $f$ is the unique increasing enumeration of a closed and unbounded (club) class of ordinals. If $f$ is a normal function, then the class
\begin{equation*}
\{\alpha\,|\,f(\alpha)=\alpha\}
\end{equation*}
of its fixed points is itself closed and unbounded. The normal function that enumerates these fixed points is called the derivative of $f$ and is denoted by $f'$. The Veblen hierarchy is a family of normal functions~$\varphi_\alpha$, indexed by the ordinals. The first function in this hierarchy is usually given as
\begin{equation*}
\varphi_0(\beta)=\omega^\beta.
\end{equation*}
Since the intersection of set-many clubs is itself a club, the function at stage $\alpha>0$ can be recursively defined by
\begin{equation*}
\varphi_\alpha:=\text{``the increasing enumeration of $\textstyle\bigcap_{\gamma<\alpha}\{\beta\,|\,\varphi_\gamma(\beta)=\beta\}$"}.
\end{equation*}
In particular we have $\varphi_{\alpha+1}={\varphi_\alpha}'$ at successor stages. Since the values of $\varphi_\alpha$ are fixed points of all previous functions in the hierarchy, we obtain
\begin{equation*}
\varphi_\gamma\circ\varphi_\alpha=\varphi_\alpha\quad\text{whenever $\gamma<\alpha$}.
\end{equation*}
It is straightforward to deduce that we have
\begin{equation}\tag{$\star$}\label{eq:veblen}
\varphi_\alpha(\beta)<\varphi_\gamma(\delta)\quad\Leftrightarrow\quad\begin{cases}
\text{either $\alpha<\beta$ and $\beta<\varphi_\gamma(\delta)$},\\
\text{or $\alpha=\gamma$ and $\beta<\delta$},\\
\text{or $\alpha>\gamma$ and $\varphi_\alpha(\beta)<\delta$}.
\end{cases}
\end{equation}
Also note that the values of $\varphi_\alpha$ are additively closed; for $\alpha>0$ they are $\varepsilon$-numbers.

Relativized notation systems $\varphi (1+X)0$ for values of the Veblen function have been described in~\cite[Definition~2.2]{rathjen-weiermann-atr} (note that our summand $1$ corresponds to the minimal element $0_Q$ that was required in the cited reference). As in the case of $\varepsilon_X$, the underlying set of $\varphi(1+X)0$ needs to be defined simultaneously with the order relation. The set $\varphi (1+X)0$ and the auxiliary function $h:\varphi (1+X)0\rightarrow 1+X$ are recursively defined by the following clauses (recall that $\bot$ denotes the unique element of $1$, which coincides with the minimal element of the order $1+X$):
\begin{itemize}
\item We have an element $0\in\varphi (1+X)0$ with $h(0)=\bot$.
\item Given elements $x\in 1+X$ and $\alpha\in\varphi(1+X)0$ with $h(\alpha)\leq_{1+X}x$, we get a term $\varphi_x\alpha\in\varphi(1+X)0$ with $h(\varphi_x\alpha)=x$.
\item Given $n>1$ elements $\varphi_{x_n}\alpha_n\leq_{\varphi(1+X)0}\dots\leq_{\varphi(1+X)0}\varphi_{x_1}\alpha_1$ of the indicated form, we get $\alpha:=\varphi_{x_1}\alpha_1+\dots+\varphi_{x_n}\alpha_n\in\varphi(1+X)0$ with $h(\alpha)=\bot$.
\end{itemize}
The order on $\varphi(1+X)0$ (which we will usually denote by $<$ rather than $<_{\varphi(1+X)0}$) reflects equivalence~(\ref{eq:veblen}), as well as the intuition that elements of the form $\varphi_x\alpha$ are additively closed. Full details can be found in~\cite[Section~2]{rathjen-weiermann-atr}. Note that we write $\varphi_x\alpha$ (without parentheses) for terms in $\varphi(1+X)0$ but $\varphi_\gamma(\alpha)$ (with parentheses) for values of the Veblen function on actual ordinals (an exception is made when parentheses in a term are needed to avoid ambiguity). In the sequel, we also write~(\ref{eq:veblen}) for the ``term version" of this equivalence in~$\varphi(1+X)0$.

Similarly to the previous section, we define a length function $L^\varphi_X:\varphi(1+X)0\rightarrow\mathbb N$ by the recursive clauses
\begin{align*}
  L^\varphi_X(0)&:=0,\\
  L^\varphi_X(\varphi_{x_1}\alpha_1+\dots+\varphi_{x_n}\alpha_n)&:=L^\varphi_X(\alpha_1)+\dots+L^\varphi_X(\alpha_n)+n,
\end{align*}
where the second clause includes the case $n=1$. We will need the following fact:

\begin{lemma}[$\rca_0$]\label{lem:subterms-ineq}
We have $\alpha<\varphi_x\alpha$  for any element $\varphi_x\alpha\in\varphi(1+X)0$.
\end{lemma}
Before we prove the lemma, let us explain how it can be reconciled with the intuition that we should have $\varphi_y\beta=\varphi_x(\varphi_y\beta)$ in case $x<_{1+X} y$. The point is that $\varphi(1+X)0$ does not even allow to form the ``superfluous" term $\varphi_x(\varphi_y\beta)$, which violates the condition $h(\varphi_y\beta)\leq_{1+X}x$.
\begin{proof}
The following stronger claim can be shown by induction on $L^\varphi_X(\alpha)+L^\varphi_X(\gamma)$:
\begin{equation*}
\text{``if $\alpha$ is a proper subterm of $\gamma\in\varphi(1+X)0$, then we have $\alpha<\gamma$."}
\end{equation*}
Let us consider the most interesting case, in which we have $\alpha=\varphi_x\beta$ and $\gamma=\varphi_y\delta$. In view of equivalence~(\ref{eq:veblen}) we need to distinguish three cases: First assume~$x<_X y$. By induction hypothesis we get $\beta<\varphi_y\delta$, which does indeed imply~$\varphi_x\beta<\varphi_y\delta$. Now assume that we have $x=y$. Since $\beta$ is a proper subterm of $\delta$, the induction hypothesis yields $\beta<\delta$. Once again~(\ref{eq:veblen}) yields the claim. Finally, assume that we have $y<_X x$. In view of
\begin{equation*}
h(\delta)\leq_Xy<_X x=h(\varphi_x\beta)
\end{equation*}
we see that $\varphi_x\beta$ cannot be equal to $\delta$; hence it must be a proper subterm. Then the induction hypothesis yields $\varphi_x\beta<\delta$, as needed to conclude by~(\ref{eq:veblen}).
\end{proof}

Above we have used $x$ to denote an arbitrary element of $1+X$. If we want to distinguish the elements of the two summands, then we write them as $\bot$ and $1+x$. On $\varphi(1+X)0$ one readily defines an operation of addition with the usual properties. Exponentiation to the base $\omega$ can be given by
\begin{equation*}
\omega^\alpha=\begin{cases}
\alpha & \text{if $\alpha=\varphi_x\beta$ with $\bot<_{1+X}x$},\\
\varphi_{\bot}\alpha & \text{otherwise}.
\end{cases}
\end{equation*}
This allows to develop a notion of Cantor normal form, which supports the usual definition of multiplication. Let us observe that values of the form $\varphi_{1+x}(\beta)$ do indeed behave like $\varepsilon$-numbers: In view of $\bot<_{1+X}1+x$ equivalence~(\ref{eq:veblen}) reveals that $\alpha<\varphi_{1+x}\beta$ implies $\omega^\alpha<\varphi_{1+x}\beta$. An element of $\varphi(1+X)0$ will be called decomposable if it is not equal to $0$ and not of the form $\varphi_{1+x}\alpha$ (hence $\varphi_\bot\alpha$ is considered as decomposable). Let us state an appropriate version of Lemma~\ref{lem:eps-nf}:

\begin{lemma}[$\rca_0$]\label{lem:phi-nf}
Any decomposable element of $\varphi(1+X)0$ can be uniquely written as $\omega^\alpha+\beta$ with $\alpha,\beta<\omega^\alpha+\beta$. Furthermore $L^\varphi_X(\alpha),L^\varphi_X(\beta)<L^\varphi_X(\omega^\alpha+\beta)$ holds for any such decomposition.
\end{lemma}
\begin{proof}
Given a decomposable element $\varphi_{x_1}\alpha_1+\dots+\varphi_{x_n}\alpha_n\in\varphi(1+X)0$, possibly with $n=1$ and $x_1=\bot$, we set
\begin{equation*}
\alpha:=\begin{cases}
\alpha_1 & \text{if $x_1=\bot$},\\
\varphi_{x_1}\alpha_1 & \text{otherwise},
\end{cases}
\end{equation*}
as well as $\beta:=\varphi_{x_2}\alpha_2+\dots+\varphi_{x_n}\alpha_n$ (in particular we have $\beta=0$ in case $n=1$). By construction we have $\omega^\alpha=\varphi_{x_1}\alpha_1$ and hence $\omega^\alpha+\beta=\varphi_{x_1}\alpha_1+\dots+\varphi_{x_n}\alpha_n$. It is straightforward to see that we have~$\alpha,\beta<\omega^\alpha+\beta$, except when we have $x_1=\bot$. In that case the claim reduces to $\alpha_1<\varphi_{\bot}\alpha_1$, which requires Lemma~\ref{lem:subterms-ineq}. The condition $L^\varphi_X(\alpha),L^\varphi_X(\beta)<L^\varphi_X(\omega^\alpha+\beta)$ is readily verified. Uniqueness follows from basic properties of addition and exponentiation, as in Lemma~\ref{lem:eps-nf}.
\end{proof}

Our goal is to characterize $\varphi(1+X)0$ as a minimal Bachmann-Howard fixed point of the transformation
\begin{equation*}
Y\mapsto T^\varphi_X(Y)=1+(Y+X)\times Y\cong 1+Y^2+X\times Y.
\end{equation*}
Elements of $T^\varphi_X(Y)$ will be written as $\bot$, $\langle y_0,y_1\rangle$ and $\langle x,y\rangle$, with $y_0,y_1,y\in Y$ and~$x\in X$. The supports from Definition~\ref{def:supports} can be given as
\begin{equation*}
\supp^\varphi_Y(\bot)=\emptyset,\qquad
\supp^\varphi_Y(\langle y_0,y_1\rangle)=\{y_0,y_1\},\qquad
\supp^\varphi_Y(\langle x,y\rangle)=\{y\}.
\end{equation*}
Hence a function $\vartheta:T^\varphi_X(Y)\rightarrow Y$ is a Bachmann-Howard collapse if, and only if, the following conditions are satisfied:
\begin{enumerate}
\item[(i)] we have $\vartheta(\bot)<_Y\vartheta(\langle y_0,y_1\rangle)$ for arbitrary elements $y_0,y_1\in Y$, as well as $\vartheta(\bot)<_Y\vartheta(\langle x,y\rangle)$ for arbitrary $y\in Y$ and $x\in X$,
\item[(i$'$)] $\langle y_0,y_1\rangle<_{Y^2}\langle y_0',y_1'\rangle$ implies $\vartheta(\langle y_0,y_1\rangle)<_Y\vartheta(\langle y_0',y_1'\rangle)$, under the side condition that we have $y_0,y_1<_Y\vartheta(\langle y_0',y_1'\rangle)$,
\item[(i$''$)] if we have $y_0,y_1<_Y\vartheta(\langle x,y\rangle)$, then we have $\vartheta(\langle y_0,y_1\rangle)<_Y\vartheta(\langle x,y\rangle)$,
\item[(i$'''$)] $\langle x,y\rangle<_{X\times Y}\langle x',y'\rangle$ implies $\vartheta(\langle x,y\rangle)<_Y\vartheta(\langle x',y'\rangle)$, under the side condition that we have $y<_Y\vartheta(\langle x',y'\rangle)$,
\item[(ii)] we have $y_0,y_1<_Y\vartheta(\langle y_0,y_1\rangle)$ for arbitrary $y_0,y_1\in Y$,
\item[(ii$'$)] we have $y<_Y\vartheta(\langle x,y\rangle)$ for arbitrary $y\in Y$ and $x\in X$.
\end{enumerate}

We can now prove the theorem that was stated in the introduction:

\begin{proof}[Proof of Theorem~\ref{thm:veblen-collapsing}]
In the first half of the proof we show that $\varphi(1+X)0$ is a Bachmann-Howard fixed point of the transformation $T^\varphi_X$. For this purpose we must specify a Bachmann-Howard collapse
\begin{equation*}
\vartheta:1+\varphi(1+X)0\times\varphi(1+X)0+X\times\varphi(1+X)0\rightarrow\varphi(1+X)0.
\end{equation*}
Above we have discussed basic ordinal arithmetic on $\varphi(1+X)0$. As in the previous section we abbreviate $\omega_2(\alpha):=\omega^{\omega^\alpha}$, as well as $1:=\omega^0$. We can now set
\begin{align*}
\vartheta(\bot)&:=0,\\
\vartheta(\langle\alpha,\beta\rangle)&:=\omega_2(\alpha+1)\cdot(\beta+1),\\
\vartheta(\langle x,\gamma\rangle)&:=\varphi_{1+x}(\gamma+1).
\end{align*}
Concerning the third clause, we observe that $\varphi_{1+x}(\gamma+1)\in\varphi(1+X)0$ holds because of~$h(\gamma+1)=\bot$ (note that $\gamma+1$ cannot be of the form $\varphi_y\delta$ with $y\neq\bot$). Terms of the form~$\varphi_\bot\delta$ are used implicitly, via the definition of exponentiation. We need to verify the conditions stated above: Conditions~(i) and~(ii) are immediate, and condition~(ii$'$) follows from Lemma~\ref{lem:subterms-ineq}. To verify condition~(i$'$) one argues just as in the proof of Theorem~\ref{thm:char-exp}. For condition~(i$''$) it suffices to recall that $\varphi_{1+x}(\gamma+1)$ behaves like an $\varepsilon$-number (cf.~also the proof of Theorem~\ref{thm:eps-fp}). In order to establish condition~(i$'''$) we consider an inequality
\begin{equation*}
\langle x,\gamma\rangle<_{X\times\varphi(1+X)0}\langle y,\delta\rangle.
\end{equation*}
If we have $x=y$ and $\gamma<\delta$, then
\begin{equation*}
\vartheta(\langle x,\gamma\rangle)=\varphi_{1+x}(\gamma+1)<\varphi_{1+y}(\delta+1)=\vartheta(\langle y,\delta\rangle)
\end{equation*}
follows from equivalence~(\ref{eq:veblen}). It remains to consider the case where we have $x<_X y$. Due to the side condition in~(i$'''$) we may assume
\begin{equation*}
\gamma<\vartheta(\langle y,\delta\rangle)=\varphi_{1+y}(\delta+1),
\end{equation*}
which can be strengthened to $\gamma+1<\varphi_{1+y}(\delta+1)$. As we also have $1+x<_{1+X}1+y$, we can again infer $\varphi_{1+x}(\gamma+1)<\varphi_{1+y}(\delta+1)$ by equivalence~(\ref{eq:veblen}). In the rest of this proof we show that the Bachmann-Howard fixed point $\varphi(1+X)0$ is minimal. For this purpose we consider an arbitrary order~$Y$ with a Bachmann-Howard collapse
\begin{equation*}
\vartheta:1+Y^2+X\times Y\rightarrow Y.
\end{equation*}
We need to construct an order embedding $f:\varphi(1+X)0\rightarrow Y$. In view of Lemma~\ref{lem:phi-nf} we can recursively define
\begin{align*}
f(0)&:=\vartheta(\bot),\\
f(\omega^\alpha+\beta)&:=\vartheta(\langle f(\alpha),f(\beta)\rangle),\quad\text{where $\alpha,\beta<\omega^\alpha+\beta$},\\
f(\varphi_{1+x}\gamma)&:=\vartheta(\langle x,f(\gamma)\rangle).
\end{align*}
Note that elements of the form $\varphi_\bot\alpha=\omega^\alpha+0$ are covered by the second clause. To show that $\eta<\xi$ implies $f(\eta)<_Y f(\xi)$ we argue by induction on $L^\varphi_X(\eta)+L^\varphi_X(\xi)$. In most cases one argues just as in the proof of Theorem~\ref{thm:eps-fp}. The only case that is essentially new concerns an inequality
\begin{equation*}
\eta=\varphi_{1+x}\gamma<\varphi_{1+z}\delta=\xi.
\end{equation*}
In view of~(\ref{eq:veblen}) we first assume that this inequality holds because we have $x<_X z$ and $\gamma<\xi$. Then we immediately get
\begin{equation*}
\langle x,f(\gamma)\rangle<_{X\times Y}\langle z,f(\delta)\rangle.
\end{equation*}
Due to the induction hypothesis we also obtain
\begin{equation*}
f(\gamma)<_Y f(\xi)=\vartheta(\langle z,f(\delta)\rangle).
\end{equation*}
This is the side condition required by~(i$'''$) above. We can thus infer
\begin{equation*}
f(\eta)=\vartheta(\langle x,f(\gamma)\rangle)<_Y\vartheta(\langle z,f(\delta)\rangle)=f(\xi).
\end{equation*}
Now assume $x=z$ and $\gamma<\delta$. The induction hypothesis yields $f(\gamma)<_Y f(\delta)$, so that we obtain $\langle x,f(\gamma)\rangle<_{X\times Y}\langle z,f(\delta)\rangle$ once again. Using Lemma~\ref{lem:subterms-ineq} we also get~$\gamma<\eta<\xi$, which allows us to conclude as in the previous case. Finally, assume that we have $z<_X x$ and $\eta<\delta$. Using the induction hypothesis and condition~(ii$'$) we obtain
\begin{equation*}
f(\eta)<_Y f(\delta)<_Y\vartheta(\langle z,f(\delta)\rangle)=f(\xi),
\end{equation*}
just as required.
\end{proof}

The statement that $\varphi(1+X)0$ is well founded for any well-order~$X$ is equivalent to the principle of arithmetical transfinite recursion, as shown by H.~Friedman (the first published proof seems to appear in~\cite{rathjen-weiermann-atr}, where a draft by Friedman, Montalb\'an and Weiermann is cited as the original reference). Now that we have proved Theorem~\ref{thm:veblen-collapsing}, we immediately obtain Corollary~\ref{cor:exp-atr} from the introduction.

In the rest of this section we show how the orders~$\Gamma_X$ can be characterized in terms of minimal Bachmann-Howard fixed points. To motivate the definition of these orders we observe that the function $\alpha\mapsto\varphi_\alpha(0)$ is normal. We write $\gamma\mapsto\Gamma_\gamma$ for the derivative of this function. Its range is the club class
\begin{equation}\tag{$\dagger$}\label{eq:gamma}
\{\Gamma_\gamma\,|\,\gamma\text{ an ordinal}\}=\{\alpha\,|\,\varphi_\alpha(0)=\alpha\}=\{\alpha>0\,|\,\varphi_\beta(\gamma)<\alpha\text{ for all }\beta,\gamma<\alpha\}
\end{equation}
of $\Gamma$-numbers. Note that any $\Gamma$-number is an $\varepsilon$-number.

A relativized notation system $\Gamma_X$ for all ordinals below the $X$-th $\Gamma$-number has been described in~\cite[Section~2]{rathjen-atr}. As for the notation system $\varphi(1+X)0$, the underlying set $\Gamma_X$ and the order relation are defined by a simultaneous recursion. In the present case we must also specify a  function $h:\Gamma_X\rightarrow\Gamma_X$ and a set $\hau\subseteq\Gamma_X$ (here $\hau$ refers to the German ``Hauptzahlen'' for (additively) principal ordinals):
\begin{itemize}
 \item We have an element $0\in\Gamma_X\backslash\hau$ with $h(0)=0$.
 \item For each $x\in X$ we have an element $\Gamma_x\in\hau\subseteq\Gamma_X$ with $h(\Gamma_x)=\Gamma_x$.
 \item Given elements $\alpha,\beta\in\Gamma_X$, we get a term $\varphi_\alpha\beta\in\hau\subseteq\Gamma_X$ with $h(\varphi_\alpha\beta)=\alpha$, provided that the following conditions are satisfied:
 \begin{itemize}
   \item we have $h(\beta)\leq_{\Gamma_X}\alpha$,
   \item if $\alpha$ is of the form $\Gamma_x$, then we have $\beta\neq 0$.
 \end{itemize}
 \item Given $n>1$ elements $\alpha_n\leq_{\Gamma_X}\dots\leq_{\Gamma_X}\alpha_1$ from $\hau\subseteq\Gamma_X$, we get a term $\alpha:=\alpha_1+\dots+\alpha_n\in\Gamma_X\backslash\hau$ with $h(\alpha)=0$.
\end{itemize}
The order on $\Gamma_X$ is determined by~(\ref{eq:veblen}) and~(\ref{eq:gamma}), where the latter are to be read as statements about terms from $\Gamma_X$ rather than actual ordinals. Full details of the somewhat lengthy definition can be found in~\cite[Section~2]{rathjen-atr}.

Addition, multiplication and exponentiation on $\Gamma_X$ can be defined as in the case of $\varphi(1+X)0$. We point out that $\varphi_0\alpha\in\Gamma_X$ assumes the role of $\varphi_\bot\alpha\in\varphi(1+X)0$. Elements of the form $\Gamma_x$ or $\varphi_\alpha\beta$ with $\alpha\neq 0$ behave like $\varepsilon$-numbers. To define a length function $L^\Gamma_X:\Gamma_X\rightarrow\mathbb N$ we set
\begin{align*}
 L^\Gamma_X(0)&:=L^\Gamma_X(\Gamma_x):=0,\\
 L^\Gamma_X(\varphi_\alpha\beta)&:=L^\Gamma_X(\alpha)+L^\Gamma_X(\beta)+1,\\
 L^\Gamma_X(\alpha_1+\dots+\alpha_n)&:=L^\Gamma_X(\alpha_1)+\dots+L^\Gamma_X(\alpha_n)+n.
\end{align*}
It will be convenient to use a somewhat different decomposition than before:

\begin{lemma}[$\rca_0$]\label{lem:Gamma-decompose}
 Any non-zero element of $\Gamma_X\backslash\hau$ can be uniquely written as~$\alpha+\beta$ with $\alpha,\beta<_{\Gamma_X}\alpha+\beta$ and $\alpha\in\hau$. Furthermore $L^\Gamma_X(\alpha),L^\Gamma_X(\beta)<L^\Gamma_X(\alpha+\beta)$ holds for any such decomposition.
\end{lemma}
\begin{proof}
 Given an element $\alpha_1+\dots+\alpha_n$ we set $\alpha:=\alpha_1$ and $\beta:=\alpha_2+\dots+\alpha_n$ (in particular $\beta=\alpha_2$ in case $n=2$). It is straightforward to see that this satisfies the desired properties. To establish uniqueness it suffices to observe that $\gamma+\alpha=\alpha$ holds for any $\gamma<_{\Gamma_X}\alpha\in\hau$ (cf.~the proof of Lemma~\ref{lem:eps-nf}).
\end{proof}

We will also need the following variant of Lemma~\ref{lem:subterms-ineq} (cf.~the explanation after the statement of that result):

\begin{lemma}[$\rca_0$]\label{lem:subterms-Gamma}
 We have $\alpha,\beta<_{\Gamma_X}\varphi_\alpha\beta$ for any element $\varphi_\alpha\beta\in\Gamma_X$.
\end{lemma}
\begin{proof}
 Yet again, the following stronger claim can be established by induction on the joint complexity $L^\Gamma_X(\eta)+L^\Gamma_X(\xi)$ of $\eta$ and $\xi$:
 \begin{equation*}
  \text{``if $\eta$ is a proper subterm of $\xi\in\Gamma_X$, then we have $\eta<_{\Gamma_X}\xi$.''}
 \end{equation*}
 Let us consider the case of $\eta=\varphi_\alpha\beta$ and $\xi=\varphi_\gamma\delta$. In contrast to Lemma~\ref{lem:subterms-ineq}, we must now distinguish two possibilities: If $\eta$ is a subterm of $\delta$, then one argues just as before. Now assume that $\eta$ is a subterm of $\gamma$. Then the induction hypothesis yields both $\alpha<_{\Gamma_X}\gamma$ and $\beta<_{\Gamma_X}\varphi_\gamma\delta$. We can conclude $\varphi_\alpha\beta<_{\Gamma_X}\varphi_\gamma\delta$ by~(\ref{eq:veblen}).
\end{proof}

Our aim is to characterize $\Gamma_X$ in terms of the order transformation
\begin{equation*}
 Y\mapsto T^\Gamma_X(Y)=1+2\times Y^2+X.
\end{equation*}
Elements of $T^\Gamma_X(Y)$ will be written as $\bot$, $\langle i,y_0,y_1\rangle$ and $x$, with $i\in\{0,1\}$, $y_0,y_1\in Y$ and $x\in X$. The supports from Definition~\ref{def:supports} amount to
\begin{equation*}
 \supp^\Gamma_Y(\bot)=\supp^\Gamma_Y(x)=\emptyset\qquad\text{and}\qquad\supp^\Gamma_Y(\langle i,y_0,y_1\rangle)=\{y_0,y_1\}.
\end{equation*}
In view of Definition~\ref{def:collapse} this means that a function $\vartheta:T^\Gamma_X(Y)\rightarrow Y$ is a Bachmann-Howard collapse if, and only if, the following conditions are satisfied:
\begin{enumerate}
 \item[(i)] we have $\vartheta(\bot)<_Y\vartheta(\langle i,y_0,y_1\rangle)$ for arbitrary $i\leq 1$ and $y_0,y_1\in Y$, as well as $\vartheta(\bot)<_Y\vartheta(x)<_Y\vartheta(x')$ for any $x,x'\in X$ with $x<_X x'$,
 \item[(i$'$)] $\langle y_0,y_1\rangle<_{Y^2}\langle y'_0,y'_1\rangle$ implies $\vartheta(\langle i,y_0,y_1\rangle)<_Y\vartheta(\langle i,y'_0,y'_1\rangle)$ for each $i\leq 1$, under the side condition that we have $y_0,y_1<_Y\vartheta(\langle i,y'_0,y'_1\rangle)$,
 \item[(i$''$)] $y_0,y_1<_Y\vartheta(\langle 1,y'_0,y'_1\rangle)$ implies $\vartheta(\langle 0,y_0,y_1\rangle)<_Y\vartheta(\langle 1,y'_0,y'_1\rangle)$,
 \item[(i$'''$)] $y_0,y_1<_Y\vartheta(x)$ implies $\vartheta(\langle i,y_0,y_1\rangle)<_Y\vartheta(x)$ for each $i\leq 1$,
 \item[(ii)] we have $y_0,y_1<_Y\vartheta(\langle i,y_0,y_1\rangle)$ for each $i\leq 1$.
\end{enumerate}

We can now establish the promised characterization:

\begin{theorem}[$\rca_0$]\label{thm:Gamma}
 The order $\Gamma_X$ is a minimal Bachmann-Howard fixed point of the transformation $T^\Gamma_X$, for each order~$X$.
\end{theorem}
\begin{proof}
 Let us first construct a Bachmann-Howard collapse
 \begin{equation*}
  \vartheta:1+2\times\Gamma_X\times\Gamma_X+X\rightarrow\Gamma_X.
 \end{equation*}
 As before we abbreviate $\omega_2(\alpha):=\omega^{\omega^\alpha}$. We then set
 \begin{align*}
  \vartheta(\bot)&:=0,\\
  \vartheta(\langle 0,\alpha,\beta\rangle)&:=\omega_2(\alpha+1)\cdot(\beta+1),\\
  \vartheta(\langle 1,\alpha,\beta\rangle)&:=\varphi_{1+\alpha}(\beta+1),\\
  \vartheta(x)&:=\Gamma_x.
 \end{align*}
 To see that we have $\varphi_{1+\alpha}(\beta+1)\in\Gamma_X$ it suffices to observe that we have $h(\beta+1)=0$ (as $\beta+1$ does not lie in $\hau\backslash\{\varphi_00\}$) and $\beta+1\neq 0$ (which is only relevant if $\alpha=\Gamma_x$). We need to show that the above conditions are satisfied. Condition~(i) is immediate. To establish condition~(i$'$) one argues just as in the proofs of Theorems~\ref{thm:char-exp} and~\ref{thm:veblen-collapsing}. Conditions~(i$''$) and (i$'''$) hold because $\varphi_{1+\alpha}(\beta+1)$ behaves like an $\varepsilon$-number (due to~$1+\alpha\neq 0$) while $\Gamma_x$ behaves like a $\Gamma$-number (cf.~equation~(\ref{eq:gamma})). Using Lemma~\ref{lem:subterms-Gamma}, one readily checks that condition~(ii) is satisfied. To show that $\Gamma_X$ is minimal we consider an arbitrary order~$Y$ with a Bachmann-Howard collapse
 \begin{equation*}
  \vartheta:1+2\times Y^2+X\rightarrow Y.
 \end{equation*}
 Relying on Lemma~\ref{lem:Gamma-decompose}, we define $f:\Gamma_X\rightarrow Y$ by the recursive clauses
 \begin{align*}
  f(0)&:=\vartheta(\bot),\\
  f(\alpha+\beta)&:=\vartheta(\langle 0,f(\alpha),f(\beta)\rangle),\quad\text{where $\alpha,\beta<\alpha+\beta$ and $\alpha\in\hau$},\\
  f(\varphi_\gamma\delta)&:=\vartheta(\langle 1,f(\gamma),f(\delta)\rangle),\\
  f(\Gamma_x)&:=\vartheta(x).
 \end{align*}
 In order to show that $\eta<_{\Gamma_X}\xi$ implies $f(\eta)<_Yf(\xi)$ one argues by induction on the number $L^\Gamma_X(\eta)+L^\Gamma_X(\xi)$. In the following we discuss the cases that are not already covered by the proofs of Theorems~\ref{thm:eps-fp} and~\ref{thm:veblen-collapsing} (where Lemma~\ref{lem:subterms-Gamma} assumes the role of Lemma~\ref{lem:subterms-ineq}). Let us first consider an inequality
 \begin{equation*}
  \eta=\varphi_\alpha\beta<_{\Gamma_X}\Gamma_x=\xi.
 \end{equation*}
 By Lemma~\ref{lem:subterms-Gamma} we get $\alpha,\beta<_{\Gamma_X}\eta<_{\Gamma_X}\xi$, so that the induction hypothesis yields
 \begin{equation*}
  f(\alpha),f(\beta)<_Yf(\xi)=\vartheta(x).
 \end{equation*}
 Using condition (i$'''$) from above we can infer
 \begin{equation*}
  f(\eta)=\vartheta(\langle 1,f(\alpha),f(\beta)\rangle)<_Y\vartheta(x)=f(\xi).
 \end{equation*}
 The case of an inequality $\alpha+\beta<_{\Gamma_X}\Gamma_x$ is treated similarly. Let us now establish the induction step for an inequality
 \begin{equation*}
  \eta=\Gamma_x<_{\Gamma_X}\varphi_\alpha\beta=\xi.
 \end{equation*}
 In view of equation~(\ref{eq:gamma}) we must have $\Gamma_x\leq_{\Gamma_X}\alpha$ or $\Gamma_x\leq_{\Gamma_X}\beta$. In either case we can invoke the induction hypothesis and condition~(ii) to get
 \begin{equation*}
  f(\eta)\leq_Y\textstyle\max_Y\{f(\alpha),f(\beta)\}<_Y\vartheta(\langle 1,f(\alpha),f(\beta)\rangle)=f(\xi).
 \end{equation*}
 A similar argument covers the case of an inequality $\Gamma_x<_{\Gamma_X}\alpha+\beta$ (where we must have $\Gamma_x\leq_{\Gamma_X}\alpha$). Finally we consider an inequality
 \begin{equation*}
  \eta=\Gamma_x<_{\Gamma_X}\Gamma_z=\xi.
 \end{equation*}
 Since $x\mapsto\Gamma_x$ represents a normal function we have $x<_X z$. By condition~(i) we get $f(\eta)=\vartheta(x)<_Y\vartheta(z)=f(\xi)$, just as required.
\end{proof}

The statement that $\Gamma_X$ is well-founded for any well-order~$X$ is equivalent to the assertion that every set lies in a countable coded $\omega$-model of arithmetical transfinite recursion, as shown by Rathjen~\cite{rathjen-atr}. We can conclude with the following:

\begin{corollary}\label{cor:model-atr}
 The following are equivalent over~$\rca_0$:
 \begin{enumerate}[label=(\roman*)]
  \item every set is contained in a countable coded $\omega$-model of $\atr$,
  \item for every well-order~$X$ the transformation $Y\mapsto 1+2\times Y^2+X$ has a well-founded Bachmann-Howard fixed point.
 \end{enumerate}
\end{corollary}

\bibliographystyle{amsplain}
\bibliography{Predicative-collapsing}

\end{document}